\documentclass{article}

\usepackage[square, sort, numbers]{natbib}
\usepackage[english]{babel}
\usepackage{newlfont}
\usepackage{latexsym}
\usepackage[utf8]{inputenc}
\usepackage{enumerate}
\usepackage{xy}
\usepackage{amsthm,amssymb,amsmath,amsopn,amsfonts,pb-diagram,accents}
\usepackage{stmaryrd,calc,enumerate,mathrsfs,amscd,wasysym,footnote}
\usepackage{graphicx}
\usepackage{fancyhdr}
\usepackage{indentfirst}
\usepackage{geometry}
\usepackage{mathtools}
\usepackage{bbm}
\usepackage{float}
\usepackage{hyperref}
\usepackage{comment}
\usepackage{forest}
\usepackage{tikz}
\usepackage{wrapfig}
\usepackage{xcolor}
\usepackage{algorithm}
\usepackage[noend]{algpseudocode}
\usepackage{cleveref}
\usepackage{tabularx}
\usepackage{soul}
\usepackage{titlefoot}

\usepackage[para]{footmisc} 
\usepackage{fixfoot} 
\DeclareFixedFootnote{\repnote}{Dipartimento di Scienze Matematiche “Giuseppe Luigi Lagrange”, Dipartimento di Eccellenza 2018-2022, Politecnico di Torino, Corso Duca degli Abruzzi 24, 10129 Torino Italy.}

\theoremstyle{plain}

\newtheorem{thm}{Theorem}[section]

\newtheorem{lem}[thm]{Lemma}
\newtheorem{prop}[thm]{Proposition}

\theoremstyle{definition}
 
\newtheorem{oss}[thm]{Remark}

\theoremstyle{remark}

\DeclareMathOperator{\argmin}{argmin}
\DeclareMathOperator{\argmax}{argmax}

\usepackage[affil-it]{authblk}

\author{Federico Santagati$^1$\footnote{federico.santagati@polito.it}}
\title{{\bf Hardy spaces on homogeneous trees \\ with  flow measures}}
\date{\vspace{-5ex}}

\begin{document}

\unmarkedfntext{{\em 2020 Mathematics Subject Classification}: 05C05; 05C21; 30H10; 35K08; 42B25; 43A99.\newline\hspace*{0.5em}{\em Keywords}: Trees; nondoubling measure; heat kernel; Hardy spaces; maximal function.\newline\hspace*{0.5em}$^1$: Dipartimento di Scienze Matematiche “Giuseppe Luigi Lagrange”, Dipartimento di Eccellenza 2018-2022, Politecnico di Torino, Corso Duca degli Abruzzi 24, 10129 Torino Italy.\newline}

\maketitle

  \begin{abstract}
 We consider a homogeneous tree endowed with a nondoubling flow measure $\mu$ of exponential growth and a probabilistic Laplacian  $\mathcal{L}$ self-adjoint with respect to $\mu$. We prove that the maximal characterization in terms of the heat and the Poisson semigroup of $\mathcal{L}$ and the Riesz transform characterization of the atomic Hardy space introduced in a previous work fail.
\end{abstract}

\section{Introduction} Let $H^1(\mathbb{R}^n)$ be the Hardy space  defined by
\begin{align*}
    H^1(\mathbb{R}^n)=\{ f \in L^1(\mathbb{R}^n) \ : \ |\nabla \Delta^{-1/2}f| \in L^1(\mathbb{R}^n)\},
\end{align*}
where $\nabla$ denotes the standard Euclidean gradient and $\Delta$ denotes the standard positive Euclidean Laplacian. 
It is a well-known fact that $H^1(\mathbb{R}^n)$ can be defined in  several equivalent ways. Indeed, a celebrated result of C. Fefferman and E.M. Stein \cite{FS, S} states the equivalence between $H^1(\mathbb{R}^n)$ and the maximal Hardy spaces defined via the heat semigroup and the Poisson semigroup of the Euclidean Laplacian. 
This result deeply depends on the doubling property of the Euclidean setting. Moreover, R. Coifman proved in \cite{Coif} that $H^1(\mathbb{R}^n)$ admits an atomic characterization. Subsequently, Coifman and G. Weiss \cite{CW} introduced an atomic Hardy space in the setting of spaces of homogeneous type; we refer to  \cite{HHLLYY, GLY, Uch, YZ} for various  maximal characterizations of the Hardy space in the context of such  spaces.  It is worth mentioning that, on   $\mathbb{R}^n$ endowed with a nondoubling  measure  of
polynomial growth, X. Tolsa  \cite{Tol03} introduced an atomic Hardy space and  proved that
it can be characterized by a  maximal operator as in the doubling setting. Furthermore, G. Mauceri and S. Meda defined an atomic Hardy space in the context of a Gaussian measure and the Ornstein–Uhlenbeck operator in $\mathbb{R}^n$ and, in the one-dimensional case, a maximal characterization of this space was proved in \cite{MMS}. Many efforts have been made in order to study nondoubling (both continuous and discrete) settings on which these characterizations fail. See for example \cite{SV2,SV, MMV2} for a contribution on a Lie group of exponential growth and on locally doubling manifolds and \cite{CM} for similar results in the context of a homogeneous tree and the combinatorial Laplacian.  \\ In \cite{LSTV} the authors define an atomic Hardy space on a tree endowed with a nondoubling, locally doubling flow measure and they prove some classical results such as the duality between $H^1$ and $BMO$ and good interpolation properties. In this paper, we focus on  the homogeneous tree $\mathbb{T}_{q+1}=(V,E)$  of order $q+1$, i.e., a tree in which every vertex has exactly $q+1$ neighbours. We consider the  metric measure space $(V, d, \mu)$ where  $d$ is the usual discrete distance on a graph and the measure $\mu$ is the canonical flow measure on a homogeneous tree (see Section \ref{lap} for a precise definition). It is worth recalling that  $(V, d, \mu)$ is of exponential growth and does not satisfy the Cheeger isoperimetric inequality (we refer to \cite{LSTV} Section 2). \\ Inspired by \cite{hs}, in Section \ref{lap} we introduce a Laplacian $\mathcal{L}$ self-adjoint on $L^2(\mu)$ that can be thought of as the natural Laplacian in this setting. We  define the heat semigroup $(\mathcal{H}_t)_{t>0}$ and the Poisson semigroup $(\mathcal{P}_t)_{t>0}$ associated with $\mathcal{L}$, given respectively  by $\mathcal{H}_t=e^{-t\mathcal{L}}$ and $\mathcal{P}_t=e^{-t\sqrt{\mathcal{L}}}$. It is a natural  task to investigate whether the Hardy spaces defined in terms of the heat semigroup and the Poisson semigroup  are equivalent to the atomic Hardy space $H^1_{at}(\mu)$ defined in \cite{LSTV} or the equivalent Hardy space defined in \cite{ATV1} (see Subsection \ref{athasp} for its definition). We define the heat maximal operator and the Poisson maximal operator as \begin{align}\label{defmht}
    \mathcal{M}_hf=\sup_{t>0}|\mathcal{H}_tf|,
\end{align} \begin{align}\label{defmpt}
    \mathcal{M}_Pf=\sup_{t>0}|\mathcal{P}_tf|,
\end{align} respectively. The aim of the first part of this work is to establish that the spaces
\begin{align*}
    &H^1_h(\mu)=\{f \in L^1(\mu) \ : \mathcal{M}_hf \in L^1(\mu)\}, \ \ \| f\|_{H^1_h}=\| f \|_1+ \| \mathcal{M}_hf\|_1,\\ 
    &H^1_P(\mu)=\{f \in L^1(\mu) \ : \mathcal{M}_Pf \in L^1(\mu)\}, \ \ \| f\|_{H^1_P}=\| f \|_1+ \| \mathcal{M}_Pf\|_1,
\end{align*} 
do not coincide with the atomic Hardy spaces $H^1_{at}(\mu)$ defined in \cite{LSTV}
(see Section \ref{athasp} for its precise definition).
 The following theorem is one of the main results of this work. It states that,  although the inclusions $H_{at}^1(\mu) \subset H^1_h(\mu), H_{at}^1(\mu) \subset H^1_P(\mu)$ are valid,  the maximal characterizations of the atomic Hardy space fail in our setting.
\begin{thm}\label{th1} i)  There exists a positive constant $C$ such that
\begin{align*}
    \|\mathcal{M}_hf\|_1 \le C \|f\|_{H^1_{at}} \qquad{\forall f \in H^1_{at}(\mu);}
\end{align*}
 ii) there exists a positive constant $C$ such that
\begin{align*}
    \|\mathcal{M}_Pf\|_1 \le C \|f\|_{H^1_{at}}  \qquad{\forall f \in H^1_{at}(\mu);}
\end{align*}
 iii) there exists a function $g \in H^1_h(\mu) \cap H^1_P(\mu)$ which does not belong to $H^1_{at}(\mu)$.
\end{thm}
It is possible to define the analogue of the Riesz transform   in our  setting, which we shall denote by $\mathcal{R}$  (see Section \ref{Riesz} for its precise definition). We introduce the Riesz Hardy space  $H^1_R(\mu)$  defined by 
\begin{align}\label{numeroRiesz}
    H^1_R(\mu)=\{f \in L^1(\mu) \ : \mathcal{R}f \in L^1(\mu)\}, 
\end{align}
which we endow with the natural norm $\ \ \| f\|_{H^1_R}=\| f \|_1+ \| \mathcal{R}f\|_1.$ \\ 

The following theorem establishes  that  the Riesz characterization of the atomic Hardy space fails. 
\begin{thm}\label{thm2}
 i)  There exists a positive constant $C$ such that
\begin{align*}
    \|\mathcal{R}f\|_1 \le C \|f\|_{H^1_{at}} \qquad{\forall f \in H^1_{at}(\mu);}
\end{align*}

 ii) there exists a function $g \in H^1_R(\mu)$ which does not belong to $H^1_{at}(\mu)$. 
\end{thm}
We point out that the  function $g$ in the above statement coincides with the function which appears in the statement of Theorem \ref{th1} $iii)$. \bigskip\\ 
This paper is organized as follows. In Section \ref{sec2}, we introduce the notation, preliminary notions and we provide useful estimates concerning the heat kernel and its gradient. Section \ref{sec3} is devoted to the proof of  Theorem \ref{th1} $i)$ and $ii)$, while in Section \ref{Heat2} we construct the function $g$ of Theorem \ref{th1} $iii).$ Finally, in Section \ref{Riesz}, we prove Theorem \ref{thm2}.
\smallskip \\ 
Along the paper, $C$ denotes a positive constant which may vary from line to line. However, when the exact values are unimportant for us, we use the standard notation $f_1(x)\lesssim f_2(x)$ to indicate
that there exists a positive constant $C$, independent of the variable $x$ but possibly depending on
some involved parameters, such that $f_1(x) \le Cf_2(x)$ for every $x$. When both $f_1(x)\lesssim f_2(x)$ and $f_2(x) \lesssim f_1(x)$ are valid, we will write $f_1(x) \approx f_2(x).$
\section{Setting}\label{sec2}
\subsection{Homogeneous trees and canonical flow}
Let  $T$  be an unoriented  tree, i.e., an unoriented connected graph  with no cycles. We denote by $V$ the set of vertices  and by $E$ the set of edges of $T$ 
and we write $x \sim y$ when $(x,y) \in E.$ If $x \sim y$ we say that $x$ is a neighbour of $y$.  Consider a sequence of vertices $\lbrace x_j\rbrace$ such that $x_j\sim x_{j+1}$. This naturally identifies an associated sequence of edges $\lbrace e_j\rbrace$, where $e_j$ is the edge connecting $x_j$ to $x_{j+1}$. We say that $\lbrace x_j\rbrace$ is a \textit{path} if $\lbrace e_j\rbrace$ does not contain repeated edges. If the path $\gamma=\lbrace x_j\rbrace_{j=0}^n$ is finite, $x_0$ and $x_n$ are called the endpoints of $\gamma$. 
The {\it discrete distance} $d(x,y)$ counts the minimum number of edges one has to cross while moving from $x$ to $y$ along a path. In a tree, for every couple of vertices $(x,y)$, there exists a unique  path   (which we call geodesic) realizing such a distance. In this case, we denote by $[x,y]$ the geodesic connecting  $x$ to $y$.  We denote by $\Gamma$ the family of geodesics and by $S_r(x)$ and $B_r(x)$  the metric sphere and  ball of center $x\in V$ and radius $r \ge 0$ respectively.  \\ Let $\mathbb{T}_{q+1}$ denote  the {\it homogeneous tree} of order $q+1$, namely, a tree such that  each vertex has exactly $q+1$ neighbours, where $q \in \mathbb{N}\setminus\{0\}.$ 
Hereinafter, we assume $T=\mathbb{T}_{q+1}=(V,E)$ with $q \ge 2.$ \\

We fix a distinguished point $o\in V$ which we call the \textit{origin} of the tree. When $x \in V$ we  denote the distance between $x$ and $o$ by $|x|$. We write $\Gamma_0$ for the family of half-infinite geodesics having an endpoint in the origin, $\Gamma_0=\lbrace \gamma=\lbrace x_j\rbrace_{j=0}^\infty\in \Gamma, x_0=o\rbrace$. The \textit{boundary} of the tree $\partial T$ is classically identified with the set of labels corresponding to elements of $\Gamma_0$,
\begin{equation*}
    \partial T=\lbrace \zeta_\gamma: \ \gamma\in \Gamma_0\rbrace.
\end{equation*}
It is standard to denote a half-infinite geodesic starting at the vertex $x$ and ending at $\xi \in \partial T$ by $[x,\xi)$. A point $\xi_0\in \partial T$ can be chosen to play the role of \textit{root} of the tree. The role of such a point is to induce a partial order relation on ${V}$. We say that $x\geq y$ if and only if $x\in [y,\xi_0)$. We define the projection of $x$ on the half-infinite geodesic $[o,\xi_0)$ as
\begin{equation*}
    \Pi_{\xi_0}(x)=\argmin_{y\in[o,\xi_0)} d(x,y),
\end{equation*}
and the \textit{level} of $x$ as
\begin{equation*}
    \ell(x)=d(o,\Pi_{\xi_0}(x))-d(\Pi_{\xi_0}(x),x).
\end{equation*}
 The fixed point $\xi_0 \in \partial T$ is called {\it mythical ancestor}. 
Note that $x\geq y$ if and only if $\ell(x)-\ell(y)=d(x,y)$. \\ 
 The \textit{predecessor} of $x$ is the unique vertex $p(x)$ such that $x\sim p(x)$ and $\ell(p(x))=\ell(x)+1$, while $y$ is a \textit{son} of $x$ if it belongs to the set $s(x)=\lbrace y\sim x: \ \ell(y)=\ell(x)-1\rbrace$. More generally, for any integer $m\ge 2$ we denote by $p^m$ the composition $p\circ p^{m-1}$, where $p^1=p.$
We define the \textit{confluent} of $x,y\in V$ as the point
\begin{equation*}
\begin{split}
    x\wedge y&=\argmax\lbrace \ell(z): \ z\in [x,y]\rbrace=\argmin\lbrace \ell(z): \ z\geq x, z\geq y\rbrace.
\end{split}
\end{equation*}

We denote by $C(V)$ the set of complex-valued functions on $V$. If $A\subset V$ we write $|A|$ to denote the cardinality of $A$. We endow $V$ with the measure $\mu$ defined as 
\begin{align*}
    \mu(A)=\sum_{x \in A}  q^{\ell(x)},
\end{align*}
where $A \subset V.$  We recall that $\mu$ is a flow measure  in the sense that 
\begin{align*}
    \mu(x)=q^{\ell(x)}=q q^{\ell(x)-1}= \sum_{y \in s(x)} \mu(y) \qquad{\forall x \in V,}
\end{align*} (see \cite{LSTV} for more information about flows). The measure $\mu$ was introduced by  W. Hebisch and T. Steger in \cite{hs} and it represents  the canonical flow measure on $T,$ since it equally distributes the mass of a vertex among its
sons. \\   For $p\in[1,\infty)$ we define $L^p(\mu)= \big\{f \in C(V) \ : \ \|f\|_p= \big(\sum_{x \in V} |f(x)|^p \ \mu(x)\big)^{1/p}<+\infty\big\}$. We also define $L^\infty(\mu)=\{f \in C(V) \ : \ \|f\|_{\infty}= \sup_{x \in V} |f(x)| <+\infty\}$ and denote by $\#$ the counting measure on $V$. Finally, if $f \in C(V)$ we define the gradient of $f$ as 
\begin{align*}
    \nabla f (x)= f(x)-f(p(x)) \qquad{\forall x \in V}.
\end{align*}
\subsection{Atomic Hardy space}\label{athasp}
 In \cite{LSTV} the authors develop a Calder\'on--Zygmund theory with respect to locally doubling flow measures and a family of sets $\mathcal{F}$ which are called admissible trapezoids. Hereinafter, we say that a set $R$ belongs to $\mathcal{F}$ if either $R=\{y\}$ for some $y \in V$ or there exist a vertex $y_R$ and two positive integers $h',h''$,  such that $R=\{y \le y_R \ : \ h' \le d(y,y_R) <h''\}=R_{h'}^{h''}(y_R)$ and  $2 \le \frac{h''}{h'} \le 12$. 

It is worth noticing that $\mu(R_{h'}^{h''}(y_R))=q^{\ell(y_R)}(h''-h').$
We introduce the atomic Hardy space $H^1_{at}(\mu)$ and its dual (we refer to \cite{LSTV} and \cite{arditti} for more details). A function $a$ is a  $(1,\infty)$-atom 
if the following hold:
  \begin{itemize}
      \item[$(i)$] $a$ is supported in a set $R\in \mathcal{F}$;  
      \item[$(ii)$] $\|a\|_\infty \le \frac{1}{\mu(R)}$;  
      \item[$(iii)$] $ \sum_{x \in R} a(x) \mu(x)=0$.
  \end{itemize}
The atomic Hardy space is defined as
\begin{align*}
    H^1_{at}(\mu)=\bigg\{f \in L^1(\mu) \ : f=\sum_{j}\lambda_j a_j, \ \{\lambda_j\}_j\in \ell^1(\mathbb{N}), \ a_j \ (1,\infty)-atom\ \bigg\},
\end{align*}
endowed with the  norm $\| f\|_{H^1_{at}}=\inf\{\sum_j |\lambda_j|: f=\sum_j\lambda_j a_j, \ a_j \ (1,\infty)-atom\}$.  \\ 
The space of functions of bounded mean oscillation is
\begin{align*}
    BMO(\mu)=\bigg\{f \in C(V) \ : \ \sup_{R \in \mathcal{F}} \frac{1}{\mu(R)}\sum_{x \in R}|f(x)-f_R| \mu(x)<+\infty\bigg\},
\end{align*}
where $f_R$ denotes the average of $f$ on $R$. \\ 
The dual of $H^1_{at}(\mu)$ can be identified with $BMO(\mu)$, see \cite[Th. 4.10]{LSTV}. In particular, if $f \in BMO(\mu)$ and $a$ is a $(1,\infty)$-atom, then
\begin{align}\label{(fiore)}
\bigg|\sum_{x \in V} f(x) a(x) \mu(x) \bigg| \lesssim \|f\|_{BMO}\|a\|_{H^1_{at}},
\end{align}

where $\|f\|_{BMO}=\sup_{R \in \mathcal{F}} \frac{1}{\mu(R)}\sum_{x \in R}|f(x)-f_R| \mu(x)$.
\subsection{Laplacians and Heat kernel}\label{lap}
Let $\Delta$ denote the combinatorial Laplacian, namely the operator defined on every $f \in C(V)$ by
\begin{align*}
    \Delta f (x)= \frac{1}{q+1} \sum_{y \sim x} (f(x)-f(y)) \qquad{\forall x \in V.}
\end{align*}
 The Laplacian $\Delta$ is bounded  on $L^p(\#)$ for any $p \in [1,\infty]$. Moreover, the $L^2(\#)$ spectrum of $\Delta$ is $[b, 2-b]$, where $b= \frac{(\sqrt{q}-1)^2}{q+1}$ (see \cite{CMS}). We refer to \cite{FTN} for more information about $\Delta$ and the spherical analysis on $T$. \\ 
 
Consider  the operator $A:C(V) \to C(V)$ defined on $f \in C(V)$ by \begin{align}\label{Stella}
    Af(x)= \frac{1}{2}\bigg(\frac{1}{q}\sum_{y \in s(x)}f(y)+f(p(x))\bigg) \qquad{\forall x \in V.}
\end{align}
Observe that we can associate to $A$ a probabilistic transition matrix, in the sense that
\begin{align}\label{proba}
  Af(x)= \sum_{y\in V} A(x,y)f(y) \ \ \text{and} \ \ \sum_{y \in V} A(x,y) =1,
\end{align} 
where $A(x,y)= \begin{cases} \frac{1}{2q} &\qquad{y \in s(x),} \\ 
\frac{1}{2} &\qquad{y=p(x),} \\
0 &\qquad{}\text{otherwise.}
\end{cases}$ \\ 

We define the operator
\begin{align}\label{deflap}
    \mathcal{L}=I-A,
\end{align}
which is the natural Laplacian in our setting. By \eqref{proba}, it is clear that $\mathcal{L}$ is a Laplacian from the probabilistic viewpoint (for more information about random walks and Laplacians on graphs  we refer to \cite{Woess}). It is also easy to see that $\mathcal{L}$ is  self-adjoint on $L^2(\mu).$  Such  operator was originally introduced  in \cite{hs}.  
It is worth noticing that
\begin{align}\label{conj}
      \mathcal{L}&=\frac{1}{1-b}\mu^{-1/2}\big(\Delta-bI\big)\mu^{1/2}.
\end{align} 
Using the fact that the pointwise multiplication by $\mu^{1/2}$ is a surjective isometry between $L^2(\#)$ and $L^2(\mu)$ and the pointwise multiplication by $\mu^{-1/2}$ is its inverse, the previous identity implies that $L^2(\mu)$-spectrum of $\mathcal{L}$ is [0,2]. \\ 
 Next, we shall define the heat kernel associated to $\mathcal{L}$ and provide some useful estimates. 
We denote by $\mathcal{H}_t$ the operator $e^{-t\mathcal{L}}$, $t >0$. Its integral kernel with respect to the measure $\mu$ is the function $H_t(\cdot,\cdot)$ such that
 for $f \in C(V)$
\begin{align*}
    \mathcal{H}_tf(x)=\sum_{y \in V} H_t(x,y)f(y)\mu(y) \qquad{\forall x \in V.}
\end{align*}

By \eqref{conj} we can explicitly write $H_t$ in terms of the heat kernel associated to $\Delta$ on $T$, which we shall denote by $h_t$.    
 By the Spectral Theorem
\begin{align}\label{awad}
   H_t(x,y)&=e^{\frac{b t}{1-b}} q^{(-\ell(y)-\ell(x))/2}h_{\frac{t}{1-b}}(x,y) \qquad{\forall t>0, x ,y \in V.} 
\end{align} Notice that, since $A$ is a transition matrix
\begin{align}\label{contraction}
    \sum_{y \in V} H_t(x,y) \mu(y) = 1, \qquad{\forall t \in \mathbb{R}^+, \ x \in V;}
\end{align} moreover, since
 $h_t(x,y)=h_t(y,x)$ we deduce that 
\begin{align*}
    H_t(x,y)=H_t(y,x) \qquad{\forall t>0, x,y \in V.}
\end{align*}

In the following, we denote by $h_t^{\mathbb{Z}}$ the heat kernel associated to the combinatorial Laplacian on $\mathbb{Z}$ and,  with a slight abuse of notation, we denote by $h^{\mathbb{Z}}_t(j)$ the function $h^{\mathbb{Z}}_t(j,0)$.\\ 
In the next proposition, we collect some results of M. Cowling, S. Meda, and A.G. Setti (see \cite[Lemma 2.4., Prop. 2.5]{CMS})  which provide an explicit expression and a sharp approximation of $h_t$ that will be useful in the sequel. 
\begin{prop}[\cite{CMS}]\label{comese} The following hold for all $t>0,$  $x \in V$ and $j \in \mathbb{N}:$
 \begin{align*}
     &i) \ h_t(x,y)= \frac{2e^{-bt}}{(1-b)t}q^{-d(x,y)/2} \sum_{k=0}^\infty q^{-k}(d(x,y)+2k+1) h^{\mathbb{Z}}_{t(1-b)}(d(x,y)+2k+1), \\ 
     &ii) \ h^{\mathbb{Z}}_t(j) \approx \frac{e^{-t+\sqrt{j^2+t^2}}}{(1+j^2+t^2)^{1/4}}\bigg(\frac{t}{j+\sqrt{j^2+t^2}}\bigg)^j, \\ 
     &iii) \ h^{\mathbb{Z}}_t(j)-h^{\mathbb{Z}}_t(j+2)=\frac{2(j+1)}{t}h^{\mathbb{Z}}_t(j+1).
 \end{align*}
\end{prop}Using $i)$ and \eqref{awad}, we easily get 
\begin{align*}
    H_t(x,y)=q^{-\ell(x)/2-\ell(y)/2}e^{bt/(1-b)}h_{t/(1-b)}(x,y)=Q(x,y)J_t(x,y),
\end{align*}
where 
\begin{align}\label{Q}
    Q(x,y)=q^{[-\ell(x)/2-\ell(y)-d(x,y)]/2}
\end{align}
and 
\begin{align}\label{J}
    J_t(x,y)=\frac{2}{t}\sum_{k=0}^\infty q^{-k}(d(x,y)+2k+1)h^{\mathbb{Z}}_t(d(x,y)+1).
\end{align} Then, by means of $i)$, we obtain the following estimate for $H_t$
\begin{align}\label{asympestimate}
    H_t(x,y) \approx \frac{Q(x,y)}{t}(d(x,y)+1)h^{\mathbb{Z}}_t(d(x,y)+1).
\end{align}

We now introduce some notation. For every $n \in \mathbb{N}$ we define the function $s_n : \mathbb{R}^+ \to \mathbb{R}$ by
\begin{align}\label{sn}
     s_n(t)= (n+1) \frac{e^{-t}e^{\sqrt{(n+1)^2+t^2}} \bigg(\frac{t}{n+1+\sqrt{(n+1)^2+t^2}}\bigg)^{n+1}}{t(1+(n+1)^2+t^2)^{1/4}} \qquad{\forall t>0.}
\end{align} Observe that by \eqref{asympestimate} and Proposition \ref{comese} $ii)$ 
\begin{align}\label{equivs}
    H_t(x,y) \approx {Q(x,y)}s_{d(x,y)}(t).
\end{align}
Let $\varphi: \mathbb{R}^+ \to \mathbb{R}$ be the function defined by 
\begin{align}\label{fi}
    \varphi(t)=-t+\sqrt{1+t^2}+ \log t-\log(1+\sqrt{1+t^2}) \qquad{\forall t>0.}
\end{align} We have that \begin{align*}
    s_n(t(n+1))=\frac{e^{(n+1)\varphi(t)}}{t(1+(n+1)^2+t^2(n+1)^2)^{1/4}}.
\end{align*}It is easy to verify that $\varphi$ is negative, increasing and 
\begin{align}\label{estfi}
   \varphi(t) \le  \frac{1}{2t}-\log\bigg(1+\frac{1}{t}\bigg)  \qquad{\forall t> 0}.
\end{align} 
We now state a  technical lemma involving the function $s_n$ defined in \eqref{sn}.
\begin{lem}\label{fun}The following hold \begin{itemize}
    \item[i)]

   $\sup_{t>0} s_n(t) \lesssim \frac{1}{(n+1)^2}.$

\item[ii)] $\sup_{t>0} \frac{n}{t}s_n(t) \lesssim \frac{1}{(n+1)^3}.$
\end{itemize}
\end{lem}
\begin{proof} 
We distinguish three different cases, namely, we estimate the supremum of the above functions when $t \ge (n+1)^2$, $n+1\le t<(n+1)^2$ and $0<t<n+1$.\\ 
{\it Case 1.}\ 
Observe that 
\begin{align*}
   \sup_{t\ge (n+1)^2}s_n(t)= \sup_{t>n+1}s_n(t(n+1))=\sup_{t>n+1}\frac{e^{(n+1)\varphi(t)}}{t[1+(n+1)^2(1+t^2)]^{1/4}}.
\end{align*} Since $\varphi$ is negative on $\mathbb{R}^+$ it follows
\begin{align*}
    &\sup_{t\ge (n+1)^2}s_n(t) \le \frac{1}{(n+1)^2} \  \ \ \text{and} \ \ \ \sup_{t \ge (n+1)^2} \frac{n}{t}s_n(t) \le \frac{1}{(n+1)^3}.
\end{align*}
{\it Case 2.} When $t \in [n+1,(n+1)^2)$ we can write $t=(n+1)\alpha $ with $\alpha \in [1,n+1)$ and 
\begin{align*}
    \sup_{n+1\le t<(n+1)^2}s_n(t)=\sup_{1\le \alpha<n+1}\frac{e^{(n+1)\varphi(\alpha)}}{\alpha[1+(n+1)^2(1+\alpha^2)]^{1/4}}.
\end{align*}
By using  \eqref{estfi} and the fact that $(1+1/\alpha)^{\alpha} \ge 2$ for all $\alpha \ge 1$, we get 
\begin{align*}
    \frac{e^{(n+1)\varphi(\alpha)}}{\alpha[1+(n+1)^2(1+\alpha^2)]^{1/4}} \le \frac{\bigg(\frac{e^{1/2}}{(1+1/\alpha)^\alpha}\bigg)^{(n+1)/\alpha}}{\alpha^{3/2}(n+1)^{1/2}}\le \frac{\bigg(\frac{e^{1/2}}{2}\bigg)^{(n+1)/\alpha}}{\alpha^{3/2}(n+1)^{1/2}}.
\end{align*}
Next, we use that $\bigg(\frac{e^{1/2}}{2}\bigg)^{(n+1)/\alpha} \lesssim \frac{\alpha^3}{(n+1)^3}$  to obtain
\begin{align*}
    \sup_{1\le \alpha<n+1}\frac{e^{(n+1)\varphi(\alpha)}}{\alpha[1+(n+1)^2(1+\alpha^2)]^{1/4}} \lesssim \sup_{1\le \alpha<n+1}\frac{\alpha^{3/2}}{(n+1)^{7/2}} \le \frac{1}{(n+1)^2}
\end{align*} and 
\begin{align*}
    \sup_{1\le \alpha<n+1} \frac{n}{(n+1)\alpha}\frac{e^{(n+1)\varphi(\alpha)}}{\alpha[1+(n+1)^2(1+\alpha^2)]^{1/4}} \lesssim \sup_{1\le \alpha<n+1}\frac{\alpha^{1/2}}{(n+1)^{7/2}} \le \frac{1}{(n+1)^3}.
\end{align*}
{\it Case 3.} In this last case $t \in (0,n+1)$ thus we can write $t=(n+1)\alpha$ with $\alpha \in (0,1)$. By using the fact that $\varphi$ is increasing and negative, we get
\begin{align*}
    s_n(\alpha(n+1))&=\frac{e^{(n+1)\varphi(\alpha)}}{\alpha[1+(n+1)^2(1+\alpha^2)]^{1/4}}\le e^{n\varphi(\alpha)}\frac{e^{\varphi(\alpha)}}{\alpha} \\ 
    &\lesssim e^{n\varphi(1)}\lesssim \frac{1}{(n+1)^2},
\end{align*}
where we have used that $\frac{e^{\varphi(\alpha)}}{\alpha} \lesssim 1$ when $\alpha \in (0,1).$ If $n=0$, then $ii)$ follows trivially. Assume $n \ge 1$ and by repeating the same argument
\begin{align*}
     \frac{n}{(n+1)\alpha}s_n(\alpha(n+1))&\lesssim e^{(n-1)\varphi(\alpha)}\frac{e^{2\varphi(\alpha)}}{\alpha^2} \lesssim \frac{1}{(n+1)^3}.
\end{align*} This concludes the proof.
\end{proof} Combining the above lemma with \eqref{equivs}, we obtain that
\begin{align}\label{ineqsquare}
    &\sup_{t>0} H_t(x,y) \lesssim \frac{Q(x,y)}{(d(x,y)+1)^2} \qquad{\forall x,y \in V,} 
    \end{align} and 
    \begin{align}\label{inescub}
    &\sup_{t>0} \frac{d(x,y)}{t}H_t(x,y) \lesssim \frac{Q(x,y)}{(d(x,y)+1)^3} \qquad{\forall x,y \in V.}
\end{align}
In the next results
we recall some pointwise and integral estimates concerning the gradient of the heat kernel which were proved in \cite{LSTV2}.
\begin{lem}\label{ktl} Assume 
$x \not\le y$ where $x,y \in V$. Then,
\begin{align*}
      &i) \ \ \ \ \ |H_t(x,y)-H_t(x,p(y))| \lesssim \max\bigg\{\frac{d(x,y)H_t(x,p(y))}{t},  \frac{H_t(x,y)}{d(x,y)+1}\bigg\}, \\ 
    &ii) \ \ \ \ \ \sup_{t>0}|H_t(x,y)-H_t(x,p(y))| \lesssim  \frac{Q(x,y)}{(d(x,y)+1)^3}.
\end{align*}
\end{lem}

\begin{proof}
 $i)$ is proved in \cite[Lemma 3.2]{LSTV2}. Combining $i)$ with \eqref{inescub}, we obtain $ii)$.
\end{proof} 
  \begin{lem}\label{lemma2}
 The following estimates hold 
  \begin{align*}
      &i) \int_1^\infty t^{-1/2} |H_t(x,y)-H_t(p(x),y)| \ dt \lesssim \frac{Q(x,y)}{(d(x,y)+1)^2} \qquad{\forall y \not \le x}, \\ 
      &ii) \ \int_{1}^\infty t^{-1/2}\frac{H_t(x,y)}{(d(x,y)+1)} \ dt \lesssim \frac{Q(x,y)}{(d(x,y)+1)^2} \qquad{\forall x,y \in V}.
  \end{align*}
  \end{lem} 
  \begin{proof}
    We refer to \cite[Lemmas 3.4, 3.5]{LSTV2} for a detailed proof.
  \end{proof}
We conclude this section with a technical lemma that provides an algorithm that we will apply to integrate a certain class of functions.
\begin{lem}\label{integ}
Let $f_{x,n}$ be the function in $C(V)$ defined by
\begin{align*}
    f_{x,n}(y)=\frac{q^{-(\ell(x)+d(x,y))/2}}{(d(x,y)+n)^2} \qquad{y \in V,}
\end{align*}
for some fixed $x \in V$ and $n \in \mathbb{N} \setminus \{0\}$.  Then, for any  $m \in \mathbb{N}\setminus \{0\}$
\begin{align*}
    \sum_{y \in S_m(x)} q^{\ell(y)/2} f_{x,n}(y) = \frac{1}{(m+n)^2} \bigg(2 +(m-1) \frac{q-1}{q}\bigg).
\end{align*}
\end{lem}
\begin{proof}
We introduce the family of sets $\{E_{m}^j\}_{j=1}^{m}, F_m$ defined by
\begin{align*}
    &E_{m}^j= S_m(x) \cap \{y \ : \ \ell(y)=\ell(x)+2j-m\}=S_m(x) \cap \{y \le p^{j}(x), y \not \le p^{j-1}(x)\}, \ \ {j=1,...,m},\\ 
    &F_m=S_m(x) \cap \{y \ : \ \ell(y) = \ell(x)-m\}=S_m(x) \cap \{y \ : \ y \le x\}.
\end{align*}
Clearly $\bigg\{\{E_{m}^j\}_{j=1}^{m}, F_m\bigg\}$ is a partition of $S_m(x)$. Moreover, $|E_m^j|=(q-1)q^{m-j-1}$ if $j <m$, $|E_m^m|=1$ and $|F_m|=q^m.$ Thus, 
\begin{align*}
     \sum_{y \in S_m(x)} q^{\ell(y)/2} f_{x,n}(y) &= \sum_{j=1}^m \sum_{y \in E_m^j} q^{(\ell(x)+2j-m)/2} f_{x,n}(y) + \sum_{y \in F_m}q^{(\ell(x)-m)/2} f_{x,n}(y)  \\ 
     &= \frac{1}{(m+n)^2}\bigg(\sum_{j=1}^{m-1} \frac{q-1}{q} + 2 \bigg).
\end{align*}
\end{proof}
\begin{oss}
The above proof illustrates the algorithm on which the computation of most of the sums throughout this paper relies. Unfortunately, although the functions we will integrate are usually of the form  $f_{x,n}$, the domain of integration might not coincide with the whole sphere  $S_m(x)$.  Thus, in each specific case, we will adapt the above idea to the particular geometry of the domain.
\end{oss}
\section{Proof of Theorem \ref{th1} i)-ii)}\label{Heat1}\label{sec3}
In this section, we shall prove that the  $L^1$-norm of the heat maximal operator $\mathcal{M}_h$ defined in \eqref{defmht} is uniformly bounded on atoms and deduce that $H^1_{at}(\mu) \subset H^1_{h}(\mu)$. By using the well-known subordination formula for the Poisson semigroup, a standard argument shows that $H_h^1(\mu) \subset H^1_P(\mu)$. Thus, Theorem \ref{th1} $ii)$  will follow immediately by Theorem \ref{th1} $i)$. \\  We preliminarily 
 need to show that $\mathcal{M}_h$ is of  weak type $(1,1)$. It is worth recalling that the weak type (1,1) boundedness of the heat maximal operator associated to the combinatorial Laplacian $\Delta$ is a well-known fact proved by M. Pagliacci and M. Picardello in \cite{pp}.\\ 
 Before establishing the abovementioned properties,  we define the local maximal heat operator by
\begin{align*}
    \mathcal{M}_{\text{loc}}f(x)= \sup_{0<t<1} |\mathcal{H}_tf(x)| \qquad{\forall f \in C(V), x \in V}.
\end{align*}
\begin{prop}\label{blmh}
The operator $\mathcal{M}_\textrm{loc}$ is bounded on $L^1(\mu).$
\end{prop}
\begin{proof}
 Let $f \in C(V)$. By \eqref{equivs} 
 \begin{align*}
      \|\mathcal{M}_{\text{loc}}f\|_1 &
       \le \sum_{y \in V} |f(y)| \sum_{x \in V} \sup_{0<t<1} H_t(x,y) \mu(x) \mu(y) \\ 
      &\lesssim \sum_{y \in V} |f(y)|\mu(y) \sum_{x \in V} \sup_{0<t<1}   Q(x,y) s_{d(x,y)}(t) \mu(x).
 \end{align*}
 It is easy to see that the term inside the second sum can be dominated as follows
 \begin{align*}
       Q(x,y) s_{d(x,y)}(t) \mu(x)
       &\lesssim q^{\frac{-d(x,y)-\ell(y)+\ell(x)}{2}} \bigg(\frac{et}{d(x,y)+1}\bigg)^{d(x,y)} \\ &\le q^{\frac{-d(x,y)-\ell(y)+\ell(x)}{2}}\bigg(\frac{e}{d(x,y)+1}\bigg)^{d(x,y)} \qquad{0<t<1.}
 \end{align*} Recalling that $\ell(x)-\ell(y) \le d(x,y)$, it suffices to notice that
 \begin{align*}
  \sum_{x \in V}  Q(x,y) \sup_{0<t<1} s_{d(x,y)}(t) \mu(x) \lesssim  \sum_{x \in V}  \bigg(\frac{e}{d(x,y)+1}\bigg)^{d(x,y)} =\sum_{d=0}^\infty\bigg(\frac{qe}{d+1}\bigg)^d < +\infty.
 \end{align*}
\end{proof}

\begin{prop} The operator $\mathcal{M}_h$ is of  weak type (1,1) and bounded on $L^p(\mu)$ for all $p \in (1,\infty].$ 
\end{prop}
\begin{proof}
It suffices to prove the weak type (1,1) boundedness of $\mathcal{M}_h$ and then use interpolation.\\  Pick $f \in L^1(\mu)$ and assume without loss of generality $f \ge 0$. Then, for every $t>0$ we have
\begin{align*}
    &\frac{1}{2t} \int_0^{2t} \mathcal{H}_z f(x) \ dz \ge \frac{1}{2t} \int_t^{2t} \mathcal{H}_z f(x) \ dz 
    =\frac{1}{2t} \sum_{y \in V} f(y) \int_t^{2t} {H}_z(x,y) \ dz \mu(y) \\&\gtrsim \frac{1}{2t} \sum_{y \in V} f(y) \int_t^{2t} Q(x,y)s_{d}(z) \ dz \mu(y),
\end{align*}
where $d=d(x,y)$. Recall that $s_d(z)=(d+1)\frac{e^{(d+1) \varphi(z/(d+1))}}{{{z}}[1+(d+1)^2+z^2]^{1/4}}$ where $\varphi$ is defined in \eqref{fi}, and 
\begin{align*}
    &\mathbb{R}^+ \ni z \mapsto e^{(d+1) \varphi(z/(d+1))} \ \ \ \text{is increasing}, \\ 
    &\mathbb{R}^+ \ni z \mapsto \frac{1}{{{{z}}[1+(d+1)^2+{{z^2}}]^{1/4}}} \ \ \ \text{is decreasing},
\end{align*}
thus
\begin{align} \label{supremum}
   \nonumber&\frac{1}{2t} \int_0^{2t} \mathcal{H}_z f(x) \ dz \gtrsim  \sum_{y \in V} f(y)  Q(x,y) \frac{(d+1)e^{(d+1) \varphi(t/(d+1))}}{{{2t}}[1+(d+1)^2+({{2t}})^2]^{1/4}}\mu(y) \\ &\gtrsim  \sum_{y \in V} f(y)H_t(x,y) \mu(y)=\mathcal{H}_tf(x),
\end{align}
where in the last line we have used \eqref{equivs}.
Observe that, by \eqref{contraction}, $(\mathcal{H}_t)_t$ is a strongly measurable semigroup which satisfies the  contraction property, namely, if $f \in L^1(\mu)$
\begin{align*}
    \|\mathcal{H}_tf\|_1\le \sum_{x \in V} \sum_{y \in V} |f(y)| H_t(x,y) \mu(y) \mu(x) =\sum_{y \in V}|f(y)|\sum_{x \in V}H_t(x,y) \ \mu(x) \ \mu(y)=\|f\|_1.
\end{align*}
Thus, by  the Hopf-Dunford-Schwartz Theorem (see \cite{Dun}), the ergodic operator associated to the heat semigroup is of  weak type (1,1).  We conclude passing to the supremum in \eqref{supremum}. \end{proof}

\begin{prop}\label{ub}
There exists a positive constant $C>0$ such that $\|\mathcal{M}_ha\|_{1} \le C$ for any   $(1,\infty)$-atom $a$.
\end{prop}
\begin{proof}
 Let $a$ be a $(1,\infty)$-atom. 
 If $\mathcal{F} \ni R=R_{h'}^{h''}(y_R)$ is the support of $a$, then we define its enlargement $R^*=\{x \in V\ : \ d(x,R) \le h'\}.$ 
 By the Cauchy-Schwarz inequality and the $L^2(\mu)$-boundedness of $\mathcal{M}_h$
 \begin{align*}
     \|\mathcal{M}_ha\|_{L^1(R^*)}\le \|\mathcal{M}_ha\|_{2}\mu(R^*)^{1/2} \le C'\|\mathcal{M}_h\|_{2\to2} \bigg(\frac{\mu(R^*)}{\mu(R)}\bigg)^{1/2} \le C,
 \end{align*}
 where we have used the fact that $\mu(R^*) \lesssim \mu(R)$, see \cite{LSTV}. \\ 
 We now split $(R^*)^c$ in two regions, namely, 
 \begin{align*}
     &\Gamma_1=\{ x \in (R^*)^c \ : \ x \le y_R\}, \\
     &\Gamma_2=(R^*)^c \setminus \Gamma_1=\{x : \ x \not\le y_R\}.  
 \end{align*}
 We start with 
 \begin{align*}
     &\sum_{x \in \Gamma_1} \mathcal{M}_ha(x) \mu(x)
     \lesssim \sum_{x \in \Gamma_1} \sup_{t>0} \sum_{y \in R}{Q(x,y)s_{d(x,y)}(t)}|a(y)|  \mu(y)\mu(x).
 \end{align*}
 By exploiting \eqref{ineqsquare} and the size condition of the atom, we get
 \begin{align*}
    \sum_{x \in \Gamma_1} \mathcal{M}_ha(x)\mu(x) 
    \lesssim \sum_{x \in \Gamma_1}  \sum_{y \in R} \frac{q^{-\ell(x)/2+\ell(y)/2-d(x,y)/2}}{(d(x,y)+1)^2} \frac{1}{\mu(R)} \mu(x).
 \end{align*} 
 If $x \in \Gamma_1$, then
 \begin{align*}
   &\frac{1}{\mu(R)} \sum_{y \in R} \frac{q^{-\ell(x)/2+\ell(y)/2-d(x,y)/2}}{(d(x,y)+1)^2} =\sum_{l=\ell(y_R)-h''+1}^{\ell(y_R)-h'}\frac{1}{\mu(R)} \sum_{y \in R \cap \{\ell(y)=l\}} \frac{q^{-\ell(x)/2+\ell(y)/2-d(x,y)/2}}{(d(x,y)+1)^2}. \end{align*}  We briefly explain how to compute the above sum. Fix $x \in \Gamma_1$ and an integer $l \in [\ell(y_R)-h''+1, \ell(y_R)-h'].$ Then, there exist  \\ $\bullet$ one vertex $ y_l \ge x$ in $R$ at level $\ell(y_l)=l$. In this case $d(x,y_l)=\ell(y_l)-\ell(x);$\\ 
 $\bullet$ $q-1$ vertices which lie at the same level as $y_l$ which belong to $U_{l,1}=\{y \ : \ell(y)=\ell(y_l), y \le p(y_l), y \ne y_l)\}$. In this case, for any $y \in U_{l,1}$,  $d(y,x)=d(y_l,x)+2$; \\ 
$\bullet$ $(q-1)q$ vertices which lie at the same level as $y_l$ which belong to $U_{l,2}=\{y \ : \ell(y)=\ell(y_l), y \le p^2(y_l), y \not\le p(y_l)\}$. In this case, for any $y \in U_{l,2}$,  $d(y,x)=d(y_l,x)+4$;\\ 
 \vdots \\ 
 $\bullet$ $(q-1)q^{d(y_l,y_R)-1}$ vertices which lie  at the same level as $y_l$ which belong to $U_{l,d(y_l,y_R)}=\{y \ : \ell(y)=\ell(y_l), y \le y_R, y \not\le p^{d(y_l,y_R)-1}(y_l)\}.$  In this case, for any $y \in U_{l,d(y_l,y_R)}$,  $d(y,x)=d(y_l,x)+2d(y_l,y_R)$. \\ 
We can rewrite the previous sum  as
  \begin{align*}
      \sum_{y \in R \cap \{\ell(y)=l\}} \frac{q^{-\ell(x)/2+\ell(y)/2-d(x,y)/2}}{(d(x,y)+1)^2} &=1 \cdot \frac{1}{(d(x,y_l)+1)^2}+\sum_{j=1}^{d(y_l,y_R)} (q-1)q^{j-1} \cdot \frac{q^{(d(y_l,x)-d(y_l,x)-2j)/2}}{(d(x,y_l)+2j+1)^2} \\ 
      &=1 \cdot \frac{1}{(d(x,y_l)+1)^2}+\sum_{j=1}^{d(y_l,y_R)} (q-1)q^{-1} \cdot \frac{1}{(d(x,y_l)+2j+1)^2} \\ 
      &\lesssim \frac{h''+h'}{(d(x,y_R)-h'')^2},
  \end{align*} 
  since $d(x,y_l)=d(x,y_R)-d(y_l,y_R) \ge d(x,y_R)-h''.$
  Summing up over the $h''-h'$ level which intersects $R$, we get
  \begin{align*}
      \frac{1}{\mu(R)} \sum_{y \in R} \frac{q^{-\ell(x)/2+\ell(y)/2-d(x,y)/2}}{(d(x,y)+1)^2}  \lesssim \frac{h''-h'}{q^{\ell(y_R)}(h''-h')} \cdot \frac{h''+h'}{(d(x,y_R)-h'')^2} \lesssim \frac{h'}{q^{\ell(y_R)}(d(x,y_R)-h'')^2}.
  \end{align*}
  We conclude that
 \begin{align*}
     &\sum_{x \in \Gamma_1} \frac{1}{q^{\ell(y_R)}} \frac{h'}{(d(x,y_R)-h'')^2} \mu(x)=\sum_{x \in \Gamma_1} \frac{q^{\ell(y_R)-d(x,y_R)}}{q^{\ell(y_R)}} \frac{h'}{(d(x,y_R)-h'')^2} \\ &\le \sum_{j \ge h'} \frac{h'}{j^2} \lesssim 1.
 \end{align*}

Now we shall integrate on $\Gamma_2$. In this case we need to use the cancellation condition of the atom. 

 It is worth noticing that the function $ R \ni y \mapsto H_t(x,y)$ with $x \in \Gamma_2$ fixed, is radial (namely, it depends only on $d(x,y)$ or equivalently, in this particular case, it depends only on $\ell(y)$). Let $y^L$ denote a vertex of maximum level in $R$. We have $d(x,y^L)=d(x,y_R)+h'$ for any $x \in \Gamma_2$.  Given a vertex $y \in R$, let $\overline{y}$ denote the predecessor of $y$ of maximum level in $R$. An easy application of Lemma \ref{ktl} and the fact that  $\ell(p^j(y))+d(x,p^j(y))=\ell(y_R)+d(x,y_R)$ for every $1 \le j \le d(y,\overline{y})$, $x \in \Gamma_2$ and $y \in R$, yield
\begin{align}\label{(cuore)}
   \nonumber \sup_{t>0}|H_t(x,y)-H_t(x,y^L)| &\le\sum_{j=0}^{d(y,\overline{y})}\sup_{t>0}|H_t(x,p^j(y))-H_t(x,p^{j+1}(y))|  \\ \nonumber &\lesssim \sum_{j=0}^{d(y,\overline{y})}  \frac{q^{-(\ell(x)+\ell(p^j(y))+d(x,p^j(y)))/2}}{(d(x,p^j(y))+1)^3} \\ 
    \nonumber &\le\sum_{j=0}^{d(y,\overline{y})}  \frac{q^{-(\ell(x)+\ell(p^j(y))+d(x,p^j(y)))/2}}{(d(x,y_R)+h')^3} \\ 
    &\le \frac{(h''-h') q^{-(\ell(x)+\ell(y_R)+d(x,y_R))/2} }{(d(x,y_R)+h')^3},
\end{align}
where in the second line we have used Lemma \ref{ktl} $ii)$ and $p^0(y)=y$. By the cancellation and the size condition of the atom and \eqref{(cuore)}
\begin{align*}
&\sup_{t>0}\bigg|\sum_{y \in R}H_t(x,y) a(y)  \mu(y)\bigg| =\sup_{t>0}\bigg|\sum_{y \in R}(H_t(x,y)-H_t(x,y^L)) a(y) \mu(y)\bigg|  \\ 
    &\le \sum_{y \in R}\sup_{t>0}|H_t(x,y)-H_t(x,y^L)| \frac{\mu(y)}{\mu(R)}  \lesssim \frac{(h''-h') q^{-(\ell(x)+\ell(y_R)+d(x,y_R))/2} }{(d(x,y_R)+h')^3}. 
\end{align*}
It follows that
\begin{align*}
    \|\mathcal{M}_ha\|_{L^1(\Gamma_2)}&=\sum_{x \in \Gamma_2}q^{\ell(x)} \sup_{t>0}\Bigg| \sum_{y \in R} H_t(x,y)a(y)\mu(y)\bigg|\\ 
    &=\sum_{x \in \Gamma_2}q^{\ell(x)} \sup_{t>0}\Bigg| \sum_{y \in R}\bigg( H_t(x,y)-H_t(x,y^L)\bigg)a(y)\mu(y)\bigg| \\ 
    &\lesssim \sum_{x \in \Gamma_2}q^{\ell(x)/2-\ell(y_R)/2-d(x,y_R)/2} \frac{(h''-h')}{(d(x,y_R)+h')^3}.
\end{align*}
We can integrate over the intersection of the spheres $S_m(y_R)$ and $\Gamma_2$ for $m \ge 1$. Arguing as in  Lemma \ref{integ} we get
\begin{align*}
    &\sum_{x \in \Gamma_2\cap S_m(y_R)}q^{\ell(x)} \sup_{t>0}\Bigg| \sum_{y \in R} H_t(x,y)a(y) \mu(y)\bigg|  \\ 
    &\lesssim 
    \frac{(h''-h')q^{-\ell(y_R)/2-m/2}}{(m+h')^3} \bigg[(q-1)\sum_{j=1}^{m-1}\bigg(q^{m-(j+1)}q^{(\ell(y_R)+2j-m)/2}\bigg)+q^{(m+\ell(y_R))/2} \bigg] \\ 
    &\lesssim \frac{(h''-h')m}{(m+h')^3} \le \frac{(h''-h')}{(m+h')^2}.
\end{align*}
Summing up over $m \ge 1$, we obtain 
\begin{align*}
    &\sum_{m=1}^\infty\sum_{x \in \Gamma_2\cap S_m(y_R)}q^{\ell(x)} \sup_{t>0}\Bigg| \sum_{y \in R} H_t(x,y)a(y) \mu(y)\bigg|  \lesssim \sum_{m=1}^\infty \frac{(h''-h')}{(m+h')^2} \lesssim 1.
\end{align*}
This concludes the proof.
\end{proof}
Using the weak type (1,1) boundedness of $\mathcal{M}_h$, it is easy to prove that the uniform boundedness of $\|\mathcal{M}_ha\|_{1}$ where $a$ ranges over $(1,\infty)$-atoms, implies the boundedness of $\mathcal{M}_h$ from $H^{1}_{at}(\mu)$ to $L^1(\mu).$ Indeed, the following can be proved by a standard argument.
\begin{lem}\label{lemgeneral}
 Let $\mathcal{K}:H^1_{at}(\mu) \to L^1(\mu)$ be a positive sublinear operator, i.e., $\mathcal{K}f \ge 0$, \ $\mathcal{K}(\alpha f)=|\alpha| \mathcal{K}(f)$  and 
 \begin{align*}
   \mathcal{K}(f+g)(x) \le \mathcal{K}(f)(x)+\mathcal{K}(g)(x), \qquad{\forall x \in V,}
 \end{align*} where $\alpha \in \mathbb{C}, f,g \in H^1_{at}(\mu)$. Suppose that there exists a positive constant $C$ such that
 \begin{align*}
     \|\mathcal{K}a\|_1 \le C,
 \end{align*}
 for all  ${(1,\infty)}$-atoms $a$. If  $\mathcal{K}$ is of  weak type (1,1), then   $$\|\mathcal{K}f\|_1 \lesssim \|f\|_{H^1_{at}} \qquad{\forall f \in H^1_{at}(\mu).}$$ 
 \end{lem}

Theorem \ref{th1} $i)$ now follows combining Proposition \ref{ub} with Lemma \ref{lemgeneral}. \\ 
We end this section with the proof of Theorem \ref{th1} $ii)$. The kernel ${P}_t(\cdot, \cdot)$ of the Poisson
semigroup $(\mathcal{P}_t)_t$ is given by the following well-known subordination formula 
\begin{align*}
  {P}_t(\cdot,\cdot)=t\int_0^\infty ({4\pi z})^{-1/2}e^{-t^2/(4z)}{H}_z(\cdot,\cdot) \ \frac{dz}{z}.
\end{align*}
We recall that the Poisson maximal operator $\mathcal{M}_P$ is defined by \eqref{defmpt}.
By a change of variables and an application of Fubini-Tonelli's Theorem, it is easily seen that $\mathcal{M}_Pf \le \mathcal{M}_hf$ for any $f \in C(V)$, thus
   $H^1_h(\mu) \subset H^1_P(\mu)$
and Theorem \ref{th1} $ii)$ is  proved.

\section{Proof of Theorem 1.1 iii)}\label{Heat2}
In this section we introduce a sequence of functions $\{g_n\}_n$ and we provide estimates of their norms in $H^1_{at}(\mu)$ and  $H^1_{h}(\mu)$. In particular, we shall obtain that $\|\cdot\|_{H^1_h}$ and $\|\cdot\|_{H^1_{at}}$ are not equivalent norms. By means of the abovementioned  estimates, we construct a function $g$ which belongs to $H^1_h(\mu)$ but which does not belong to $H^1_{at}(\mu).$ Exploiting the inclusion $H^1_{h}(\mu) \subset H^1_P(\mu)$, we will obtain also that $g \in H^1_P(\mu)$. 
 \smallskip\\  We introduce a numeration on the set of vertices of level $0$ as follows.  For all $n \ge 2$ if  $\ell(x)=0$, $x \le p^n(o)$ and $x \not \le p^{n-1}(o)$ we assign to $x$ a unique label $x_i$ with $i  \in [q^{n-1},q^n-1]$. If $x \le p(o)$, then we define $x_0=o$ and the remaining $q-1$ vertices $x_i$ with $i=1,...,q-1$. \\ Define \begin{align}\label{gn} g_n(x)=\delta_{x_n}(x)-\delta_o(x)  \qquad{\forall n \ge 2.}\end{align} Since  $g_n$ is supported in $\{x_n\}\cup\{o\}$ and has zero average for every $n \ge2$, it follows that $g_n \in H^1_{at}(\mu)$.  In order to estimate  $\|g_n\|_{H^1_{at}}$ from below,  we shall construct a function $f \in BMO(\mu)$ and apply \eqref{(fiore)}. Consider the function $f:V \to \mathbb{R}$ defined as follows
\begin{align}\label{fs} f(x)=\begin{cases}n\log q &\text{if $x \le p^n(o), \ x \not \le  p^{n-1}(o), \ \text{and} \  n \ge 2,$} \\   \log q &\text{if $x \le p^{1}(o)$.} \end{cases} \end{align}
\begin{prop}
The function $f$ defined  by \eqref{fs} belongs to $BMO(\mu).$
\end{prop}
\begin{proof}
 It is easy to see that $f$ is constant on every admissible trapezoid with root not in $[p^2(o),\xi_0)$.
Hence, to prove that $f \in BMO(\mu)$ we have to control the average of $f$ on an admissible trapezoid $R$ with root  in $[p^2(o),\xi_0)$. We claim that it suffices to prove the uniform boundedness of \begin{align*}
    \frac{1}{\mu(R)}\sum_{x \in R} |f(x)-C_R|\mu(x),
\end{align*}  where $C_R$ is a suitable constant depending only on $R$. Indeed, for any $y \in R$
\begin{align*}
    |f(y)-f_R|\le |f(y)-C_R|+|C_R-f_R| \le |f(y)-C_R|+\frac{1}{\mu(R)}\sum_{x \in R} |f(x)-C_R|  \ \mu(x),
\end{align*}
and it follows 
\begin{align*}
   \frac{1}{\mu(R)}\sum_{y \in R} |f(y)-f_R|\mu(y)\le \frac{2}{\mu(R)} \sum_{y \in R} |f(y)-C_R| \mu(y),
\end{align*}
and the last inequality proves the claim. \\ 
Next, we distinguish two cases.
\\{\it Case 1.} Let $R=R_{h'}^{h''}(p^{(n)}(o))$ with  $n \ge h''.$ We shall estimate from above \begin{align*}  &\frac{1}{\mu(R)}\sum_{x \in R} |f(x)-n\log q|\mu(x).
   \end{align*}
  Using the definition of $f$, it is convenient to compute the above sum on each level. Indeed, fix a positive integer $l \in [n-h''+1,n-h']$. Then,
  \begin{align*}
    &\frac{1}{\mu(R)}\sum_{x \in R \cap \ell(x)=l} |f(x)-n\log q| \mu(x)\\ &=\frac{q^{l}}{\mu(R)}\bigg[\bigg((q-1)\sum_{j=l+1}^{n}q^{j-1-l}|j\log q-n\log q|\bigg)  +1\cdot|l\log q -n\log q|\bigg]\\ 
    &\le\sum_{j=l}^{n-1}\frac{q^{l}}{\mu(R)}q^{j-l}(n-j)\log q \\ 
    &=\sum_{j=l}^{n-1}q^{j-n}\frac{(n-j)}{(h''-h')}\log q \\ 
    &\le \sum_{m=1}^{\infty}q^{-m}\frac{m}{h''-h'}\log q.
  \end{align*}
  We get an estimate independent of $l$. Summing over the $h''-h'$ levels which intersect $R$, we conclude that
  \begin{align*}
      \frac{1}{\mu(R)}\sum_{x \in R} |f(x)-n\log q|\mu(x)&=\sum_{l=n-h''+1}^{n-h'}\frac{1}{\mu(R)}\sum_{x \in R \cap \ell(x)=l} |f(x)-n\log q|\mu(x)  \\ &\le (h''-h') \sum_{m=1}^{\infty}q^{-m}\frac{m}{(h''-h')}\log q \\ &\lesssim 1.
  \end{align*} 
{\it Case 2.} Let $R=R_{h'}^{h''}(p^n(o))$ with  $2 \le n<h''$. We can follow the previous argument except for the levels $l \le 0$. Thus, if $0 \ge l \in [n-h''+1,n-h']$ is a fixed level, 
\begin{align*}
    &\frac{1}{\mu(R)}\sum_{x \in R \cap \ell(x)=l} |f-n\log q|\mu(x) \\&=\frac{q^{l}}{\mu(R)}\bigg[\bigg((q-1)\sum_{j=2}^{n}q^{j-1-l}|j\log q-n\log q|\bigg)+q^{1-l}|\log q -n\log q|\bigg]\\
     &\le \sum_{j=1}^{n-1}\frac{q^{l}}{\mu(R)}q^{j-l}(n-j)\log q,
\end{align*}
and we conclude as above. \\ 

This proves that $f \in BMO(\mu)$. 
\end{proof}
\begin{oss}\label{remark} If we take  $n$ such that $q^{m-1} \le  n \le q^m-1$ for $m \ge 2$, then it is easily seen that $|x_n| = 2m \le 2\frac{\log n}{\log q}+2 \lesssim  \log n$, while $f(x_n)= m \log q \ge \log n$.  \\ We also underline that $x_n \wedge o=p^{|x_n|/2}(o)=p^{|x_n|/2}(x_n)$ for all $n \ge 2$.  \end{oss}
 Since $g_n$ is a multiple of a $(1,\infty)-$atom, by \eqref{(fiore)} we get  
\begin{align*}
   \|f\|_{BMO}\|g_n\|_{H^1_{at}} \gtrsim \bigg|\sum_{x \in V} f(x)g_n(x)\mu(x) \bigg| =|f(x_n)-f(o)| \gtrsim \log n,
\end{align*}
which implies that
\begin{align}\label{numero}
    \log n \lesssim \|g_n\|_{H^1_{at}}.
\end{align}
Moreover, it is clear that $\|g_n\|_{1} \approx 1.$
Combining the previous inequalities with the following proposition we conclude that the norms on $H^1_h(\mu)$ and $H^1_{at}(\mu)$ are not equivalent.
\begin{prop}\label{ros2} Let $\{g_n\}_n$ be the sequence defined in \eqref{gn}. Then, the following holds:
    \begin{align*}
        \|\mathcal{M}_hg_n\|_1 \lesssim \log \log n \qquad{\forall n \ge 2.}
    \end{align*}
\end{prop}
\begin{proof}
We split the proof into three steps. \\ 
{\it Step 1.} \\ Define  $B=B(o,|x_n|)$. Our goal is to show that 
\begin{align*}
    \sum_{x \in B} \mathcal{M}_h\delta_{x_j}(x) \mu(x)\lesssim \log \log n
\end{align*}
for $j=0$ and $j=n$. 

 Notice that  for all $x \in V$,  by \eqref{ineqsquare}
\begin{align}\label{fiore}
    &\mathcal{M}_h(\delta_{x_j})(x)\mu(x) = \mu(x)\sup_{t>0} H_t(x,x_j) 
    \lesssim  \frac{Q(x,x_j) \mu(x)}{(d(x,x_j)+1)^2}.
\end{align}
By \eqref{fiore}
\begin{align*}
    \sum_{x \in B} \mathcal{M}_h(\delta_o)(x) \mu(x) &\lesssim \sum_{x\in B} q^{\ell(x)/2}\frac{q^{-|x|/2}}{(|x|+1)^2}. 
\end{align*}We write $B =\cup_{m=0}^{|x_n|} S_m(o)$ and apply Lemma \ref{integ}  to obtain 
 \begin{align}\label{laprima}
         \sum_{x \in B} \mathcal{M}_h(\delta_o)(x) \mu(x)\lesssim \sum_{m=0}^{|x_n|} \frac{1}{m+1} \lesssim \log |x_n| \lesssim \log \log n,
 \end{align}
 where we refer to Remark \ref{remark} for the last estimate.  \\ It remains to prove the same inequality which involves $\mathcal{M}_h(\delta_{x_n})$. Again by \eqref{fiore}
\begin{align*}
    \sum_{x \in B}\mathcal{M}_h(\delta_{x_n})(x)\mu(x)= \sum_{x \in B}\mu(x)\sup_{t>0}|H_t(x,x_n)| \lesssim \sum_{x \in B}\frac{Q(x,x_n) \mu(x)}{(d(x,x_n)+1)^2}.
\end{align*}
Denote by $B^*$ the ball $B(x_n,2|x_n|)$. Clearly, $B \subset B^*$. Hence 
\begin{align*}
      \sum_{x \in B}\mathcal{M}_h(\delta_{x_n})(x)\mu(x)\lesssim \sum_{x \in B^*}\frac{q^{\ell(x)/2} q^{-d(x,x_n)/2}}{(d(x,x_n)+1)^2}.
\end{align*} Exactly as in \eqref{laprima} we get 
\begin{align}\label{numero777}
     \sum_{x \in B}\mathcal{M}_h(\delta_{x_n})(x)\mu(x)\lesssim \sum_{m=0}^{2|x_n|} \frac{1}{m+1} \lesssim \log 2|x_n| \lesssim \log \log n.
\end{align}

This is the desired conclusion. \\ 

 {\it Step 2.} \\ We divide the complement of $B(o,|x_n|)$ in two regions.  
\begin{align*}
    &\Gamma_1=\{x \in B(o,|x_n|)^c \ : \ x\le  p^{|x_n|}(o)\}, \\ 
    &\Gamma_2=\{x \in B(o,|x_n|)^c \ : x \not\in \Gamma_1\}. 
\end{align*}
We claim that 
    \begin{align}\label{doppiastella}
        \sum_{x \in \Gamma_1} \mathcal{M}_h(\delta_o)(x) \mu(x) \lesssim 1.
    \end{align}

 The claim follows by a direct computation. 
Indeed,  we estimate the above sum on $S_m(o) \cap \Gamma_1$ for every $m >|x_n|$ as follows
 \begin{align}{\label{form}}
    \nonumber &\sum_{x \in S_m(o) \cap \Gamma_1} \mathcal{M}_h(\delta_o)(x) \mu(x) \lesssim \sum_{x \in S_m(o) \cap \Gamma_1} \frac{q^{\ell(x)/2-d(x,o)/2}}{(d(x,o)+1)^2}   \\ 
     &=\frac{q^{-m/2}}{m^2}\bigg[(q-1)\sum_{j=1}^{|x_n|}\bigg(q^{m-(j+1)}q^{(2j-m)/2}\bigg) +q^{m/2}\bigg] \lesssim \frac{|x_n|}{m^2}
 \end{align}
 where we integrate adapting  Lemma \ref{integ}. 
 We conclude by observing that 
 \begin{align*}
    \sum_{x \in \Gamma_1}\mathcal{M}_h(\delta_o)(x) \mu(x)&=\sum_{m =|x_n|+1}^\infty\sum_{x \in S_m(o) \cap \Gamma_1} q^{\ell(x)}\mathcal{M}_h(\delta_o)(x) \lesssim \sum_{m =|x_n|+1}^\infty\frac{|x_n|}{m^2} \lesssim 1,
 \end{align*}
 
     and  \eqref{doppiastella} is proved. \\ 
     We now claim that
     \begin{align}\label{triplastella}
         \sum_{x \in \Gamma_1} \mathcal{M}_h(\delta_{x_n})(x) \mu(x) \lesssim \log \log n.
     \end{align}
     For establishing  it, in order to exploit the symmetries of $\mathcal{M}_h(\delta_{x_n})$, it is convenient to  integrate on a larger set than $\Gamma_1$. Define $\Gamma_1^{*}=\{y \in V \ : \ y \not\le p^{|x_n|/2}(o) \}$
\        and observe that if $x \in \Gamma_1\cap \Gamma_1^{*}$ then $d(x_n,x)=d(o,x)$, (because $x_n \wedge o= p^{|x_n|/2}(x_n)=p^{|x_n|/2}(o)$),  thus  $\mathcal{M}_h(\delta_{x_n})(x)=\mathcal{M}_h(\delta_o)(x)$. Obviously $\Gamma_1 \cap ({\Gamma_1^{*}})^c \subset ({\Gamma_1^{*}})^c=\{y \in \ : y \le p^{|x_n|/2}(o)\}.$ \\ It suffices to check that
      \begin{align*}
          \sum_{ x \in ({\Gamma_1^{*}})^c} \mathcal{M}_h(\delta_{x_n})(x) \mu(x) \lesssim  \log \log n.
      \end{align*}
  It is convenient to think of the above sum as the sum over the disjoint sets $\{S_m(x_n) \cap (\Gamma_1^{*})^c\}_{m \ge 0}$. Fix $m \ge 0$ and by applying \eqref{ineqsquare} we obtain
  \begin{align*}
      \sum_{x\in S_m(x_n) \cap (\Gamma_1^{*})^c} \mathcal{M}_h(\delta_{x_n})(x) \mu(x)\lesssim  \sum_{x\in S_m(x_n) \cap (\Gamma_1^{*})^c} q^{\ell(x)/2} \frac{q^{-d(x,x_n)/2}}{(d(x,x_n)+1)^2}.
  \end{align*}
  Assume $m>|x_n|/2$. In the same fashion as we computed in Lemma $\ref{integ}$, we obtain 
  \begin{align*}
      &\sum_{x\in S_m(x_n) \cap {(\Gamma_1^{*})^c}} \mathcal{M}_h(\delta_{x_n})(x) \mu(x)  \\ 
      &\lesssim 
    \frac{q^{-m/2}}{m^2}\bigg[ (q-1)\sum_{j=1}^{|x_n|/2} \bigg(q^{m-(j+1)}q^{(2j-m)/2}\bigg) + q^{m/2}\bigg] \\ 
      &\lesssim \frac{|x_n|/2}{m^2}.
  \end{align*}

If $m<|x_n|/2$, the same computation still works with a slight modification, 
  \begin{align*}
        &\sum_{x \in S_m(x_n) \cap {(\Gamma_1^{*})^c}} \mathcal{M}_h(\delta_{x_n})(x) \mu(x) \lesssim\frac{q^{-m/2}}{m^2}\bigg[ q^{m/2}+ (q-1)\sum_{j=1}^{m-1} \bigg(q^{m-(j+1)}q^{(2j-m)/2}\bigg) + q^{m/2}\bigg] \lesssim \frac{1}{m},
  \end{align*}
  where the first term inside the square brackets is the contribution due to $p^{m}(x_n) \in (\Gamma_1^{*})^c$.
  Summing up over the positive integers, we conclude 
  \begin{align*}
      \sum_{m=1}^\infty \sum_{x \in S_m(x_n) \cap {(\Gamma_1^{*})^c}} \mathcal{M}_h(\delta_{x_n})(x) \mu(x) &\lesssim \sum_{m=1}^{|x_n|/2-1} \frac{1}{m} + \sum_{m=|x_n|/2}^{\infty} \frac{|x_n|/2}{m^2} \\ 
      &\lesssim \log(|x_n|)+ 1 \lesssim \log\log n,
  \end{align*}
 which proves \eqref{triplastella}. \\ 
   
{\it Step 3.} \\ Notice that, if $x \not\le x_n \wedge o=p^{|x_n|/2}(o)$, then $d(x_n,x)=d(x,o)$. This is true because, for such a vertex $x$ 
\begin{align*}
    d(x,o)=d(x,x_n \wedge o)+d(x_n \wedge o,o)=d(x,x_n \wedge o)+d(x_n \wedge o,x_n)=d(x,x_n).
\end{align*}
Observe that this together with \eqref{awad} imply 
\begin{align}\label{pokerstella}
    \nonumber&\sum_{x \in \Gamma_2}  \mathcal{M}_h(\delta_{x_n}-\delta_o)(x) \mu(x) \\ &= \frac{1}{1-b}\sum_{x \in \Gamma_2} q^{\ell(x)/2} \sup_{t>0} {e^{bt/(1-b)}}|q^{-\ell(x_n)/2}h_{t/(1-b)}(x,x_n)-q^{-\ell(o)/2}h_{t/(1-b)}(x,o)| \ dx=0,
\end{align}
since $q^{\ell(x_n)}=q^{\ell(o)}=1$ and $h_{t/(1-b)}(x,y)=h_{t/(1-b)}(d(x,y))$. \\ 
In conclusion, \eqref{laprima},  \eqref{numero777}, \eqref{doppiastella}, \eqref{triplastella} and \eqref{pokerstella} yield
\begin{align*}
    \|\mathcal{M}_h g_n\|_1 \lesssim \log \log n.
\end{align*}
\end{proof}
It follows that \begin{align*}
   \lim_{n \to \infty}  \frac{\|g_n\|_{H^1_h}}{\|g_n\|_{H^1_{at}}} = 0,
\end{align*}
and in particular, $\| \cdot \|_{H^1_{h}}, \ \|\cdot \|_{H^1_{at}}$ are not equivalent. \\ We are now ready to prove Theorem \ref{th1} $iii)$. 
\begin{proof}[Proof of Theorem \ref{th1} iii)] Define the function $g$ on the set of vertices at level $0$  as $g(o)=c_0$,  $g(x)=0$ if $x \le p^{1}(o)\setminus \{o\}$ and $g(x_n)= \frac{1}{n (\log n)^{3/2}}$ for every $n \ge q$.  Then we extend $g=0$ outside the level zero. Choose $c_0$ such that $\sum_{x \in V} g(x) \mu(x)=0.$ Clearly, 
    \begin{align*}
        \|g\|_1= |c_0|+\sum_{n=q}^\infty \frac{1}{n(\log n)^{3/2}} <+\infty.
    \end{align*}  We now show that $\|\mathcal{M}_hg\|_1$ is finite. Indeed, we observe that
    \begin{align*}
        g=\sum_{k=q}^\infty c_k g_k, 
    \end{align*}
    where $\{g_k\}_k$ is defined  in \eqref{gn} and $c_k$ is the value of $g$ at $x_{k}$.  Then, by using Proposition \ref{ros2}
    \begin{align*}
        \|\mathcal{M}_hg\|_1\lesssim \sum_{k=q}^\infty c_k \log \log k \lesssim \sum_k \frac{\log\log k}{k(\log k)^{3/2}}<+\infty.
    \end{align*} This implies that $g \in H^1_h(\mu).$ \\ 
    We now prove that $g \not \in H^1_{at}(\mu)$. Indeed, suppose the converse by contradiction. Then it would be 
    \begin{align}\label{false}
        \sum_{x \in V} g(x) f(x) \mu(x)<+\infty,
    \end{align}
    where $f$ is the BMO function defined in  \eqref{fs}. But using the estimate $f(x_n) \ge \log n $ (see Remark \ref{remark}), \eqref{false} would imply 
    \begin{align*}
        \sum_{n=q}^\infty \frac{1}{n(\log n)^{1/2}}<+\infty,
    \end{align*}
    which is clearly false. Then $g \not \in H^1_{at}(\mu).$
\end{proof}

  \section{Proof of Theorem \ref{thm2}}\label{Riesz} This last section is devoted to the proof of Theorem \ref{thm2}. We briefly recall some preliminary notion. \\\  
   We define the discrete Riesz transform  $\mathcal{R}= \nabla \mathcal{L}^{-1/2}$,
    which corresponds to the integral operator with integral kernel  with respect to $\mu$ $$R(x,y)=\int_0^{+\infty} t^{-1/2} (H_t(x,y)-H_t(p(x),y))\ dt.$$ 
    
   Recall that the Riesz Hardy space is defined by \eqref{numeroRiesz}. It is a well-known fact that $\mathcal{R}$ maps $H^1_{at}(\mu)$ to $L^1(\mu)$, indeed, it is an easy consequence of the discrete version of H\"ormander's condition for singular operators (see \cite[Th. 3]{ATV1} or \cite{LSTV} and \cite{hs}). Thus, the inclusion $H^1_{at}(\mu) \subset H^1_{R}(\mu)$ is trivial. 
        \\ In order to show that such inequality is strict, we need the following result.  
        \begin{prop}\label{r0gn} The following holds
        \begin{align*}
            \|\mathcal{R}g_n\|_1 \lesssim \log \log n \qquad{\forall n \ge 2,}
        \end{align*}
            where $\{g_n\}_n$ is the sequence defined in \eqref{gn}.
        \end{prop}

  \begin{proof}
  We write 
  \begin{align*}
      R(x,y)&=\int_0^1 t^{-1/2} (H_t(x,y)-H_t(p(x),y))\ dt+ \int_1^{\infty} t^{-1/2} (H_t(x,y)-H_t(p(x),y))\ dt \\ &=R^{(0)}(x,y)+R^{(\infty)}(x,y)
  \end{align*}
  and consequently $\mathcal{R}=\mathcal{R}^{(0)}+\mathcal{R}^{(\infty)}$.
It follows from Proposition \ref{blmh} that  $\mathcal{R}^{(0)}$ is bounded on $L^1(\mu),$ hence $\|\mathcal{R}^{(0)}g_n\|_1 \lesssim 1.$ We now consider $\|\mathcal{R}^{(\infty)}g_n\|_1$. We recall that
    \begin{align*}
        \|\mathcal{R}^{(\infty)}g_n\|_1&=\sum_{x \in V} \bigg| \sum_{y \in V} \int_1^\infty t^{-1/2} (H_t(x,y)-H_t(p(x),y)) \ dt g_n(y) \mu(y) \bigg|\mu(x) \\ 
        &= \sum_{x \in V}  \bigg|\int_1^\infty t^{-1/2} (H_t(x,x_n)-H_t(x,o)+H_t(p(x),o)-H_t(p(x),x_n)) \ dt  \bigg|  \mu(x).
    \end{align*}
Arguing as in  Step 3 of Proposition \ref{ros2}, we get that, if $x \not \le x_n \wedge o$, the first difference inside the integral in the last line vanishes. The same happens for the second difference if  $p(x) \not \le x_n \wedge o$. Since $$\{x \in V \ : x \not\le x_n \wedge o\} \subset \{x \in V \ : \ p(x)  \not\le x_n \wedge o\},$$ we can estimate the previous sum as follows
    \begin{align*}
        &\|\mathcal{R}^{(\infty)}g_n\|_1 
       \\&\le\sum_{x \in E_n} \int_1^\infty \frac{|H_t(x,x_n)-H_t(p(x),x_n)|}{t^{1/2}} \ dt\ \mu(x)+ \sum_{x \in E_n}\int_1^{\infty}\frac{|H_t(x,o)-H_t(p(x),o))|}{t^{1/2}} \ dt \ \mu(x)\\ 
        &= I_1+I_2,
    \end{align*} where $E_n=\{x \in V \ : \ x \le x_n\wedge o\}.$ Observe that $E_n=\Gamma_1\cup\Gamma_2=\Sigma_1\cup\Sigma_2,$ where
    \begin{align*}
        &\Gamma_1=\{x \in E_n \ : x_n \not \le x\}, \\ 
         &\Gamma_2=\{x \in E_n \ : \   x_n \le x
        \}, \\ &\Sigma_1=\{x \in E_n  \ : o \not \le x\}, \\ 
        &\Sigma_2=\{x \in E_n \ : o \le x\}.
    \end{align*}We start studying $I_1$. Exploiting the symmetry of the problem, the same computations are valid for $I_2$. It can be useful to split the sum which defines $I_1$ as 
    \begin{align*}
     I_1=\sum_{i=1}^2\sum_{x \in \Gamma_i}\int_1^{\infty}\frac{|H_t(x,x_n)-H_t(p(x),x_n))|}{t^{1/2}} \ dt \ \mu(x)=I_1^1+I_1^2.
    \end{align*}
    By Lemma \ref{lemma2} $i)$,
    \begin{align*}
        I_1^1 \lesssim \sum_{x \in \Gamma_1} \frac{Q(x,x_n)}{(d(x,x_n)+1)^2} \mu(x).
    \end{align*}
     Since $x_n \wedge o=p^{|x_n|/2}(x_n),$ we can think of the sum on $\Gamma_1$ as the sum on the sequence of disjoint sets $\{\Gamma_1^j\}_{j=0}^{|x_n|/2}$,  where $\Gamma_1^j$ is defined by
     \begin{align*}
        \Gamma_1^j= \begin{cases}  \{x \le x_n\} &\text{if $j=0$}, \\ 
         \{x \le p^j(x_n) \ {\text{and}} \ x \not\le p^{j-1}(x_n)\} &\text{if $1 \le j\le |x_n|/2$},
         \end{cases}
     \end{align*}
    with $p^0(x_n)=x_n.$ 
     Observe that, for any $j=1,...,|x_n|/2$, $x \in\Gamma_1^j$ implies that
    \begin{align*}
    d(x,x_n)=2j-\ell(x),
    \end{align*}
    where we have used that $\ell(p^{j}(x_n))=j$. Then, for any $1 \le j \le |x_n|/2$ 
    \begin{align*}
        &\sum_{x \in \Gamma_1^j} \frac{ q^{\ell(x)/2-d(x,x_n)/2}}{(d(x,x_n)+1)^2} \le \sum_{l=-\infty}^j q^{l-j} \frac{1}{(2j-l)^2} (q-1)q^{j-l-1}     \le \frac{2}{j},
    \end{align*}
    where $(q-1)q^{j-l-1}$ corresponds to the cardinality of vertices in $\Gamma_1^j$ at the level $l$. The sum over $\Gamma_1^0$ contributes to the sum as a constant independent of $n$. Summing up 
    \begin{align*}
        I_1^1 \lesssim \sum_{j=1}^{|x_n|/2} \frac{1}{j} \lesssim \log \log n.
    \end{align*}
     It remains to estimate $I_1^2$.
   \  By Lemma \ref{lemma2} $ii)$ and the fact that if $x \in \Gamma_2$, then $\ell(x)=d(x,x_n)$ and $$Q(x,x_n)=qQ(p(x),x_n)=q^{-d(x,x_n)},$$ we get
    \begin{align*}
        I_1^2 &\le \sum_{x \in \Gamma_2}  \int_1^\infty t^{-1/2} \max\{H_t(x,x_n),H_t(p(x),x_n)\} \ dt \ \mu(x) \\&\lesssim \sum_{x \in \Gamma_2}  \frac{q^{-d(x,x_n)}}{d(x,x_n)+1} \mu(x) = \sum_{d=1}^{|x_n|/2}  \frac{1}{d} \lesssim \log \log n.
    \end{align*}
  Similar computations can be repeated to estimate $I_2$ if we replace $\Gamma_i$ by $\Sigma_i$. In conclusion
    \begin{align*}
        \|\mathcal{R}g_n\|_1 \lesssim \log \log n,
    \end{align*}
    as required.
      \end{proof}
      We conclude the proof of Theorem \ref{thm2}.
    \begin{proof}[Proof of Theorem \ref{thm2} ii)] Let $g$ be the function constructed in the proof of Theorem \ref{th1} $iii)$. Then, 
    \begin{align*}
        \|\mathcal{R}g\|_1 \lesssim \sum_{k=q}^\infty c_k\|\mathcal{R}g_k\|_1\lesssim \sum_{k=q}^\infty \frac{\log \log k}{k(\log k)^{3/2}}<+\infty.
    \end{align*}
 Hence $g \in H^1_{R}(\mu)$ but $g \not\in H^1_{at}(\mu).$
    \end{proof}
    \begin{oss}It is not clear whether $H^1_h(\mu),H^1_P(\mu)$ and $H^1_R(\mu)$ are the same space or not. This is an interesting open problem that we
have not been able to answer and leave for further work.    
    \end{oss}
    \bigskip
	
	{\bf{Acknowledgments.}} The author would like to thank Maria Vallarino for fruitful
conversations and comments. The author is very grateful to the anonymous referee for
her/his helpful report.
	Work partially supported by the MIUR project ``Dipartimenti di Eccellenza 2018-2022" (CUP E11G18000350001) and the Progetto GNAMPA 2020 ``Fractional Laplacians and subLaplacians on Lie groups and trees". 
	
	The author is member of the Gruppo Nazionale per l'Analisi Matema\-tica, la Probabilit\`a e le loro Applicazioni (GNAMPA) of the Istituto Nazionale di Alta Matematica (INdAM).
    
    \bibliographystyle{plain}
{\small
\bibliography{references}}

\begin{thebibliography}{10}

\bibitem{arditti}
Laura Arditti, Anita Tabacco, and Maria Vallarino.
\newblock {BMO} spaces on weighted homogeneous trees.
\newblock {\em J. Geom. Anal.}, pages 1--18, 2020.

\bibitem{ATV1}
Laura Arditti, Anita Tabacco, and Maria Vallarino.
\newblock {H}ardy spaces on weighted homogeneous trees.
\newblock {\em Advances in Microlocal and Time-Frequency Analysis}, pages
  21--39, 2020.

\bibitem{CM}
Dario Celotto and Stefano Meda.
\newblock On the analogue of the {F}efferman-{S}tein theorem on graphs with the
  {C}heeger property.
\newblock {\em Ann. Mat. Pura Appl. (4)}, 197(5):1637--1677, 2018.

\bibitem{Coif}
Ronald~R. Coifman.
\newblock A real variable characterization of ${H}^p$.
\newblock {\em Studia Math.}, 51:269--274, 1974.

\bibitem{CW}
Ronald~R. Coifman and Guido Weiss.
\newblock Extensions of {H}ardy spaces and their use in analysis.
\newblock {\em Bull. Amer. Math. Soc.}, 83(4):569--645, 1977.

\bibitem{CMS}
Michael Cowling, Stefano Meda, and Alberto~G. Setti.
\newblock Estimates for functions of the {L}aplace operator on homogeneous
  trees.
\newblock {\em Trans. Amer. Math. Soc.}, 352(9):4271--4293, 2000.

\bibitem{Dun}
Nelson Dunford and Jacob~T. Schwartz.
\newblock {\em Linear Operator Part I: General Theory}, volume VII of {\em Pure
  and applied mathematics (Interscience publishers)}.
\newblock 1958.

\bibitem{FS}
Charles Fefferman and Elias.~M. Stein.
\newblock {$H^{p}$} spaces of several variables.
\newblock {\em Acta Math.}, 129(3-4):137--193, 1972.

\bibitem{FTN}
Alessandro Figà-Talamanca and Claudio Nebbia.
\newblock {\em Harmonic analysis and representation theory for groups acting on
  homogenous trees}, volume 162.
\newblock Cambridge University Press, 1991.

\bibitem{GLY}
Loukas Grafakos, Liguang Liu, and Dachun Yang.
\newblock Radial maximal function characterizations for {H}ardy spaces on
  {RD}-spaces.
\newblock {\em Bull. Soc. Math. Fr.}, 137:225--251, 2009.

\bibitem{HHLLYY}
Ziyi He, Yongsheng Han, Ji~Li, Liguang Liu, Dachun Yang, and Wen Yuan.
\newblock A complete real-variable theory of {H}ardy spaces on spaces of
  homogeneous type.
\newblock {\em J. Fourier Anal. Appl.}, 25:2197--2267, 2019.

\bibitem{hs}
Waldemar Hebisch and Tim Steger.
\newblock Multipliers and singular integrals on exponential growth groups.
\newblock {\em Math. Z.}, 245(1):37--61, 2003.

\bibitem{LSTV}
Matteo Levi, Federico Santagati, Anita Tabacco, and Maria Vallarino.
\newblock Analysis on trees with nondoubling flow measures.
\newblock {\em Potential Anal.}, 2021.

\bibitem{LSTV2}
Matteo Levi, Federico Santagati, Anita Tabacco, and Maria Vallarino.
\newblock Riesz transform for a flow {L}aplacian on homogeneous trees.
\newblock {\em ArXiv:2107.06620}, 2021.

\bibitem{MMV2}
Alessio Martini, Stefano Meda, and Maria Vallarino.
\newblock Maximal characterisation of local {H}ardy spaces on locally doubling
  manifolds.
\newblock {\em Math. Z.}, 2021.

\bibitem{MMS}
Giancarlo Mauceri, Stefano Meda, and Peter Sj\"ogren.
\newblock A maximal function characterization of the {H}ardy space for the
  {G}auss measure.
\newblock {\em Proc. Amer. Math. Soc.}, 141(5):1679--1692, 2013.

\bibitem{pp}
Mauro Pagliacci and Massimo Picardello.
\newblock The heat diffusion on homogeneous trees.
\newblock {\em Adv. Math.}, 110:175--190, 1995.

\bibitem{SV2}
Peter Sj\"ogren and Maria Vallarino.
\newblock Boundedness from ${H}^1$ to ${L}^1$ of {R}iesz transforms on a {L}ie
  group of exponential growth.
\newblock {\em Ann. de l'Institut Fourier}, 58(4):1117--1151, 2008.

\bibitem{SV}
Peter Sj\"ogren and Maria Vallarino.
\newblock Heat maximal function on a {L}ie group of exponential growth.
\newblock {\em Ann. Acad. Sci. Fenn. Math}, 37(3):491--507, 2012.

\bibitem{S}
Elias~M. Stein.
\newblock {\em Harmonic analysis: real-variable methods, orthogonality, and
  oscillatory integrals}, volume~43 of {\em Princeton Mathematical Series}.
\newblock Princeton University Press, Princeton, NJ, 1993.

\bibitem{Tol03}
Xavier Tolsa.
\newblock The space ${H}^1$ for nondoubling measures in terms of a grand
  maximal operator.
\newblock {\em Trans. Amer. Math. Soc.}, 355:315--358, 2003.

\bibitem{Uch}
Akihito Uchiyama.
\newblock A maximal function characterization of {$H^p$} on the space of
  homogeneous type.
\newblock {\em Trans. Amer. Math. Soc}, 262:579--592, 1980.

\bibitem{Woess}
Wolfgang Woess.
\newblock {\em Random walks on infinite graphs and groups}.
\newblock Cambridge University Press, 2000.

\bibitem{YZ}
Dachun Yang and Yuan Zhou.
\newblock Radial maximal function characterizations of {H}ardy spaces on
  {RD}-spaces and their applications.
\newblock {\em Math. Ann.}, 346:307--333, 2010.

\end{thebibliography}
\end{document}